\documentclass[a4paper,11pt]{article}

\usepackage[latin1]{inputenc}%Lettere accentate
\usepackage[english]{babel}%Stile inglese
\usepackage{amsmath,amsfonts,amssymb,amsthm,bbm}%Ambienti matematici
\usepackage{units}
\usepackage{float}
\floatstyle{boxed} 
\restylefloat{figure}
\usepackage{graphicx}
\usepackage{caption}
\usepackage{subcaption}

\theoremstyle{definition}

\newtheorem{lemma}{Lemma}[section]
\newtheorem{prop}[lemma]{Proposition}
\newtheorem{thm}[lemma]{Theorem}
\newtheorem{remark}[lemma]{Remark}
\newtheorem{example}[lemma]{Example}
\newtheorem{defin}[lemma]{Definition}

\renewenvironment{proof}[1][]{\par \textbf{Proof #1.}}{\par\medskip\hfill\hbox{$\quad\Box$}\par\medskip\ignorespacesafterend}

\newenvironment{subtitle}{\itshape}{\par\medskip}

\newcommand{\R}{\mathbb{R}}

\newcommand{\N}{\mathbb{N}}
\newcommand{\Q}{\mathbb{Q}}
\newcommand{\Z}{\mathbb{Z}}
\newcommand{\F}{\mathbb{F}}

\begin{document}

\begin{center}
{\def\baselinestretch{1.5}
\large \bf Hybrid Continued Fractions and $n$-adic algorithms,\\
\vspace{2mm}
\normalsize{with applications to cryptography and ``unimaginable'' numbers}}
\vspace{6mm}

{\normalsize\bf Antonino Leonardis$^1$}
\vspace{10mm}

{\scriptsize
$^1$ Department of Mathematics and Computer Science, University of Calabria \\ 
Arcavacata di Rende, Italy \\
e-mail: {antonino.leonardis@unical.it}
\vspace{4mm}
}
\end{center}
\vspace{10mm}

\section*{Abstract}
This paper continues the author's previous studies on continued fractions and Heron's algorithm, as from his former JMM2017 presentation (see \cite{CF.HA}).\par\medskip
Extending the notion of continued fraction to the $p$-adic fields, one can find continued fractions which converge in both real and $p$-adic topologies to the ``same'' quadratic irrational number, some of which are given by the Heron's algorithm using a generalized version of an author's theorem from the cited JMM presentation. The definition can be possibly generalized to other global fields, as left as an open question. We will end the part on hybrid convergence with many numerical examples. After that, we will recall the basic algorithms on the $p$-adic fields studied by the author and see some applications of theirs to computer science: applying Heron's algorithm to quickly compute $p$-adic square roots, finding new elementary cryptography procedures and some methods to get pseudo-random numbers, calculate last digits of some peculiar very big numbers.

\section{Introduction}
Heron's algorithm is very important in history of mathematics, and is a particular case of Newton's method for approximating zeroes of a function. We will recall in this work a theorem on Heron's algorithm related to continued fractions proved in \cite{CF.HA}. The theorem will also hold for a generalization of continued fractions (definition \ref{def:generalized.CF}) which converges also in local $p$-adic completions (for every prime divisor $p$ of a fixed natural squarefree number $n$), in the wake of the author's previous work (see PhD thesis \cite{CF.NA}). After recalling Lagrange's theorem for quadratic irrationals, we will see that some of them have a common representation both as real and $p$-adic ($p$ as before) numbers. We must recall that Heron's algorithm is also very important in $p$-adic fields, where Newton's method is better known as the fundamental \emph{Hensel's lemma} and gives a very fast algorithm for computing the $p$-adic square root (compare \cite{SML}). Other generalizations of continued fractions may be studied for such hybrid convergences, and this is left as an open problem. In the end of the section we will leave an open question over possible generalizations of the definition to other fields and then give some numerical examples of the hybrid convergence in the generalized definition given here. In the last section we will see applications of general $n$-adic algorithms, such as pseudo-randomizing methods (see \cite{MM.GF}, \cite{RND}) and fast computation of the last decimal digits for some ``unimaginable numbers'' (see \cite{UN}).\par\medskip
This work has been presented both at ``The First Symposium of the International Pythagorean School -- da Pitagora a Sch\"utzenberger: numeri inimmaginabil\^i\^i\^i'' held in Cosenza, Italy (september 2018) and at the JMM 2019 held in Baltimore, MD.

\subsection{$p$-adic fields and $n$-adic rings}
\begin{subtitle}See \cite{pAdic} and \cite{CASS} for details on the theory of $p$-adic numbers.\end{subtitle}
Given a prime number $p$, the set of $p$-adic integers can be represented as a power series in the letter $p$ with exponents growing towards $+\infty$ and digits in a finite set of $p$ elements, generally $0,1,\ldots,p-1$. More generally, one can allow $p=n$ to be any given integer $n>1$.\par\medskip
For instance, the number $-1\in\Z$ can be represented in the $10$-adic ring as:
$$\ldots999999=\sum_{i=0}^\infty{9\cdot 10^i}$$
because indeed one expects that (as we will see in the ``$n$-adic algorithms'' section):
\begin{align*}
	\ldots999999&+\\
	1&=\\
	\ldots000000&
\end{align*}
In the prime case $n=p$ one obtains a field $\Q_p$, which is the completion of rational numbers under the so-called $p$-adic absolute value. The latter is defined for a fraction $\pm p^t\frac ab$ (for uniquely determined $t\in\Z$, $a,b\in\Z_{>0}$ not divisible by $p$) as:
$$\left|\pm p^t\frac ab\right|_p=p^{-t}$$
so that in fact positive powers of $p$ converge to $0$.\par\medskip
In the general case $\Z_n$ is a ring with the following properties:
\begin{itemize}
	\item $\Z_{p^k}\overset{\cong}{\longrightarrow}\Z_p$ for any $k>1$ by writing all digits of the first ring in base $p$.
	\item $\Z_{mn}\overset{\cong}{\longrightarrow}\Z_m\times\Z_n$ when $m$ and $n$ are coprime as approximants converge in both rings, giving this way the two projections, and the map is invertible by chinese remainder theorem.
	\item Thus $\Z_n\cong\prod_{i=1}^k \Z_{p_i}$ considering all prime divisors $p_i$ of $n$. As a side fact, when $k>1$ the ring has always zero-divisors.
\end{itemize}
Structure of the multiplicative group $\Z_p^\times$ is discussed in \cite{CF.NA}.
\begin{example}
	Hexadecimal representations of $24$-bits $RGB$ colors is just a shortened version of the actual bit representation:
	\begin{align*}
		\text{\ttfamily FF·FF·00}_{16}&=1111'1111'1111'1111'0000'0000_2\text{ (yellow)}\\
		\text{\ttfamily 00·FF·7F}_{16}&=0000'0000'1111'1111'0111'1111_2\text{ (water blue)}\\
		\text{\ttfamily 7F·3F·00}_{16}&=0111'1111'0011'1111'0000'0000_2\text{ (brown)}
	\end{align*}
\end{example}
\begin{example}
	Let's see an example of a decimal zero-divisor (as explained in the ``last digits of unimaginable numbers'' section):
	\begin{align*}
	\ldots 896109004106619977392256259918212890625.&\times\\
	\ldots 896109004106619977392256259918212890624.&=\\
	\ldots 000000000000000000000000000000000000000.&
	\end{align*}
\end{example}
\subsection{Heron's algorithm and Newton's method}
The \textbf{Newton's method} for finding the roots of a polynomial consists in starting with an approximation $x_i$ of such a root, drawing the tangent at the point $(x_i,y_i)$ to the graph and intersecting this tangent with the $x$-axis, obtaining $x_{i+1}$; the sequence obtained by this method converges to a root of the polynomial.\par\medskip
When one considers the polynomial $x^2-a$ and starts with $x_0=\lfloor\sqrt{a}\rfloor$, Newton's method just sends $x_i$ to the arithmetic mean between $x_i$ and $\frac{a}{x_i}$ and is known as the \textbf{Heron's algorithm} for finding (the geometric mean) $\sqrt{a}$. The intervals between $x_i$ and $\frac{a}{x_i}$ give a so-called \textbf{sequence of chinese boxes}.\par\medskip
We recall the main theorem from \cite{CF.HA}:
\begin{thm}[AL 2017]
\label{thm.AL2017}
Suppose $\xi$ has (classical) continued fraction expansion $[n_0=\left\lfloor\sqrt{x}\right\rfloor,n_1,n_2,\ldots]$ with period of length $1$ or $2$. Apply Heron's algorithm to $a_0:=n_0$ obtaining a sequence $\{a_0,a_1,\ldots\}$. Then $a_i$ is the $2^i$-th approximant via the continued fraction. Vice versa, if the continued fraction has period length greater than $2$, applying in the same way Heron's algorithm one does not get the same sequence of approximants.
\end{thm}
\section{Continued fractions which converge in different absolute values}
\begin{subtitle}
It is possible to define continued fractions in such a way that they converge in $\Q_p$, but the definition is far from being unique and often depends on the results one wants to achieve (apart from author's work, see also \cite{CF.local.fields.1} and \cite{CF.local.fields.2}). We will now consider a definition which ensures convergence for more than just one absolute value. Heron's algorithm in $p$-adic fields will be a particular case of Hensel's lemma.
\end{subtitle}
\begin{defin}
\label{def:generalized.CF}
We will consider a non-zero integer $n\in\Z$, which is specifically the product of $k$ different primes ($n=\pm p_1^{\alpha_1}p_2^{\alpha_2}\cdots p_k^{\alpha_k}$). Then we may consider expressions of the form:
$$x=[a_0;a_1,a_2,\ldots]_n=a_0+\frac{n}{a_1+\frac{n}{a_2+\frac{n}{\vdots}}}$$
where we consider integers $a_i\in\Z$ such that $a_i\geq |n|$ and $(a_i,n)=1$ (inequality is obviously strict when $|n|>1$ from the coprimality condition). We may accept $a_0$ not to be restricted to values greater than $|n|$ as one does with usual continued fractions. We remark that $1$-continued fractions are the classical ones, while $-1$-continued fractions are the almost-classical ones obtained using the ``ceiling'' integral part.
\end{defin}
\begin{thm}
\label{thm:hybridCF}
The $n$-continued fraction defined above converges with respect to the $p_i$-adic absolute value for any prime $p_i$ dividing $n$ and also with respect to the usual absolute value in the real numbers.
\end{thm}
\begin{proof}[\ref{thm:hybridCF}]
With the usual methods for continued fractions in local fields (see \cite{CF.NA}) one can prove that it converges in all $p_i$-adic fields (indeed, fixed a prime, the other primes dividing $n$ are invertible in $\Z_p$ so can be moved to the denominator of the fractions), or equivalently in the $n$-adic ring $\Q_n$. Moreover, we have:
$$[a_0;a_1,a_2,\ldots]_n=a_0+\frac{1}{a_1/n+\frac{1}{a_2+\frac{1}{\vdots}}}=b_0+\frac{1}{b_1+\frac{1}{b_2+\frac{1}{\vdots}}}$$
and this gives the convergence in the set of real numbers in a way analogous to the standard continued fractions.\par\medskip
As a brief recall, we consider the recurrence step for the approximants $\frac{p_i}{q_i}=\left[b_0;b_1,b_2,\ldots,b_i\right]$:
$$\left(\begin{array}{cc}p_k&p_{k-1}\\q_k&q_{k-1}\end{array}\right)\left(\begin{array}{cc}b_{k+1}&1\\1&0\end{array}\right)=\left(\begin{array}{cc}b_{k+1}p_k+p_{k-1}&p_k\\b_{k+1}q_k+q_{k-1}&q_k\end{array}\right)$$
from which we get the following facts:
\begin{itemize}
	\item the denominator recurrence $q_{k+1}=b_{k+1}q_k+q_{k-1}\geq q_k+q_{k-1}$ is necessarily exponentially divergent ($b_{k+1}\geq1$);
	\item the difference of two consecutive approximants $\frac{p_{k+1}}{q_{k+1}}-\frac{p_k}{q_k}$ is as well exponentially convergent to zero because simple calculations give:
	$$\frac{(b_{k+1}p_k+p_{k-1})q_k-p_k(b_{k+1}q_k+q_{k-1})}{q_{k+1}q_k}=\frac{p_{k-1}q_k-p_kq_{k-1}}{q_{k+1}q_k}=\frac{\pm1}{q_{k+1}q_k};$$
	\item thus, the sequence of approximants satisfies Cauchy's condition and converges to a uniquely determined real number.
\end{itemize}
The following results are straightforward:
\begin{prop}
Theorem \ref{thm.AL2017} about Heron's algorithm holds also for $p$-adic continued fraction expansions. Specifically, the proof is the same without supposing $\lambda$ being an integer and considering the $b_i$ expansion. More precisely, given a non-square positive integer $x$, the fact that $\sqrt{x}$ has a hybrid $n$-continued fraction expansion of the form $[a;\overline{b,2a}]_n$ is necessary and sufficient condition for each $a_i$ to be the $2^i$-th approximant via the continued fraction.
\end{prop}
\begin{prop}
Periodic hybrid continued fractions converge to the ``same'' algebraic (quadratic) number for each absolute value which is considered, meaning that it is the solution of the same $2^\text{nd}$ degree equation over $\Q$.
\end{prop}
\end{proof}
We conclude this section with some numerical examples after the following:
\begin{remark}
Continued fractions may be considered also in other global fields and their completions (for instance with Gauss integers $\Z+i\Z$ and complex numbers), with the due attentions. An interesting open question may be whether there is a similar hybrid convergence theory in such cases.
\end{remark}
\subsection{Examples}
\label{sect:examples}
\begin{enumerate}
	\item The first example is $[6,6,\ldots]=[\overline{6}]_5$; a simple calculation shows that $[\overline{6}]_5=3+\sqrt{14}$ meaning that subtracting $3$ we obtain a square root of $14$. It converges in both the usual real absolute value and in the $5$-adic one. Also, $3$ is an integral part for $\sqrt{14}$ both in the real field and in $\Z_5$ as $14\equiv 3^2$ modulo $5$. Thus the Heron's algorithm gives the continued fraction approximants (and consequently converges) in both cases:
	\begin{align*}
	a_0&=3&=[3]_5;\\
	a_1&=\frac{3}{2}+\frac{7}{3}=\frac{23}{6}&=[3,6]_5;\\
	a_2&=\frac{23}{6\cdot2}+\frac{7\cdot6}{23}=\frac{1033}{276}&=[3,6,6,6]_5.
	\end{align*}
	\item We give another perfectly similar example. We take $4+\sqrt{21}=[\overline{8}]_5$, whose radical part has Heron's approximants:
	\begin{align*}
	a_0&=4&=[4]_5;\\
	a_1&=\frac{4}{2}+\frac{21}{8}=\frac{37}{8}&=[4,8]_5;\\
	a_2&=\frac{37}{8\cdot2}+\frac{21\cdot8}{37\cdot 2}=\frac{2713}{592}&=[4,8,8,8]_5.
	\end{align*}
	\item In the following example we consider a simple composite $n$, i.e. $n=10$. Here we have $\frac{11+\sqrt{161}}2=[\overline{11}]_{10}$ with first approximants $11,\frac{131}{11},\frac{1551}{131},\ldots$ in $\R,\Q_2,\Q_5$. Here it is not possible to have a square root without denominator so we can't relate to the Heron's algorithm of the square root of an integer.
	\item In the case $n=15$ we may have instead an integral expression such as $[\overline{16}]_{15}=8+\sqrt{79}$ which converges in $\R,\Q_3,\Q_5$. This expression relates to Heron's algorithm for $\sqrt{79}$:
	\begin{align*}
	a_0&=8&=[8]_{15};\\
	a_1&=\frac{143}{16}&=[8,16]_{15};\\
	a_2&=\frac{40673}{4576}&=[8,16,16,16]_{15}.
	\end{align*}
	\item We conclude this part with a length-2 period case. We take $4+\sqrt{22}=[\overline{8,4}]_3$, whose radical part has Heron's approximants:
	\begin{align*}
	a_0&=4&=[4]_3;\\
	a_1&=\frac{4}{2}+\frac{22}{8}=\frac{19}{4}&=[4,4]_3;\\
	a_2&=\frac{19}{4\cdot2}+\frac{22\cdot4}{19\cdot 2}=\frac{713}{152}&=[4,4,8,4]_5.
	\end{align*}
\end{enumerate}

\section{$n$-adic algorithms and applications}
\begin{subtitle}Calculations with $n$-adic numbers can be used for cryptography and for generating random numbers, as we will see in this section.\end{subtitle}
\subsection{Basic results}

\begin{remark}
In the case of $p$-adic fields, and more generally $n$-adic rings (where $n=p$ is not necessarily a prime number) the equivalent of Newton's approximation method of tangents is known as one of the most important theorems under the name of Hensel's lemma. Thus the convergence of Heron's algorithm can be studied using well known results about Hensel's lemma, and one finds out that a precision of $p^{-k}$ is always achieved within $\log_2 k$ steps. Thus it has applications to $p$-adic computations, where it is indeed a very fast algorithm for computing the square root of a number (compare corollary 4.2.3 from \cite{SML}).
\end{remark}

We recall that $p$-adic numbers, and more in general $n$-adic numbers for any given integer $n>1$, can be represented as a power series in the letter $p$ (or $n$) with exponents growing towards $+\infty$ and digits in a finite set of $p$ (or $n$) elements, generally $0,1,\ldots,p-1$ (the same as real numbers, but with possibly infinite digits only on the left of the decimal point, reflecting the different convergence). In this setup, we can implement algorithm for calculation with the following proposition (see \cite{SML}):

\begin{prop}
The fundamental operations (addition, opposite, multiplication, reciprocal) can be turned into an easy algorithm for calculations into the set $\Z_p^\times$ ``$p$-adically approximated to the $k$-th digit'' by the projection (or \emph{truncation}) into $(\Z/p^k\Z)^\times$, i.e. modulo any element with $p$-adic absolute value less than $p^{-k}$. A similar result holds for $n$-adic rings, given any integer $n>1$. The algorithms can be (and have been by the author) implemented on a computer.
\end{prop}
Moreover:
\begin{prop}
Any advanced operation which requires approximations by Hensel's lemma (provided that the input satisfies lemma's hypothesis) can be as well implemented as an algorithm, as already done by the author.\par\smallskip
Some common examples are:
\begin{itemize}
	\item quadratic surds;
	\item exponentials and logarithms (or other functions determined by a power series) within their convergence radius.
\end{itemize}
\end{prop}
\subsection{Elementary criptography methods}
Let us consider the strings of $k$ elements in the set $0,1,\ldots,n-1$, i.e. elements of $S=\Z/n^k\Z$. Fixing an element $y\in(\Z/n^k\Z)^\times$ (it suffices that the last digit is coprime to $n$), this determines a simple symmetric encrypting key $S\rightarrow S$ by normal $n$-adic multiplication:
$$x\rightarrow xy$$
where the decrypting key is simply $y^{-1}$.
It is likely that our further studies of these kind of algorithms will lead us to better encryption methods.
\begin{example}
Decimal numbers from $0000$ to $9999$ can be encrypted using the key $73$, which has as a reverse key the number $137$ by the well known peculiar factorization $73\times 137=10001$.
\end{example}
\begin{example}
Alphanumeric strings can be considered as $37$-adic numbers with digits $0, \ldots, 9, \text{A}, \ldots, \text{Y}, \text{Z}, \text{\_}$ (the last being a separator) and therefore encrypted using the properties of $\F_{37}=\Z/37\Z$.
\end{example}
\begin{example}
Using bit-mapped bases, we may as well consider the possibility to apply our cryptography methods to such expressions.\par\medskip
We know that using basis $2^{100\times 100}$ (which corresponds to representing digits with a $100\times100$ black and white bitmap), or any other power of $2$, a number is invertible if and only if the last digit is odd; the only possible issue in the division algorithm is effectively calculating the reciprocal of this last digit, but this can be done by considering for this digit the $2$-adic expansion and applying $2$-adic division.\par\medskip
Representing digits as b/w bitmaps has many possible practical applications, as the technology to let an electronic device read such digits is already widespread (barcodes, QR-codes).
\end{example}
\subsection{Randomizing algorithms}
We can try the following randomizing methods given a seed $s\in\N$ (and ignoring trivial endings):
\begin{itemize}
	\item transforming $p$-continued fractions into sums and vice-versa (digits $0,\ldots,p-1$);
	\item considering $\sqrt{1+ps}$ or other $p$-converging power series;
	\item iterating $s$ times the function $x_i\to x_{i+1}=\sqrt{x_i}$ in any $n$-adic ring.
\end{itemize}
Indeed, the operation of going from natural numbers to $p$-adic numbers and vice-versa is totally unpredictable and as such satisfies the 3rd law of randomness (see \cite{RND}).\par\medskip
\begin{remark}
The Montecarlo method for calculating a probability, which is used in physics abundantly, can be used in reverse to verify that an experiment is random enough. Instead of determining the average value of the outcomes for an experiment in order to get an approximated expected value of its distribution, one checks that the known expected value is comparable with the statistically obtained one. We will specifically use the fact that in an $n\times n$ square the probability to have a distance from a corner less than $n$ is approximately $\pi/4$.
\end{remark}
We consider the third possibility, applying our fast algorithm for computing square roots. We show some heuristic data giving credit to the fact that it generates pseudo-random digits.\par\medskip
We start by creating a list of random couples $(x,y)$ with $0<=x<=15624=444444_5$ and $0<=y<=15624=444444_5$. Grouping together $40$ of those couples, we count how many satisfy $x^2+y^2<=15625$, and multiply the ratio by $4$.\par\smallskip
We iterate $100$ times this procedure and obtain a statistic with average $3.115$ and variance $0.1$, confirming the pseudo-randomness of the algorithm by the 3rd law of pseudo-randomness (\emph{cannot be distinguished from a random one by a chosen method}).
\subsection{Last digits of ``unimaginable numbers''}
For some so-called ``unimaginable'' numbers (see \cite{UN}), like Graham's number, we may calculate the last digits using $n$-adic numbers.\par\medskip
\begin{example}
In $\Z_{10}$ there are two non trivial idempotent elements (whose sum is $1$ and product is $0$) given by:
\begin{align*}
16^{5^\infty}&:=\lim_{k\to\infty}16^{5^k}=\ldots 07743740081787109376;\\
5^{2^\infty}&:=\lim_{k\to\infty}5^{2^k}=\ldots 92256259918212890625.
\end{align*}
Their last $k$ digits are in common with the unimaginable numbers $16^{5^k}$ and $5^{2^k}$, given that $k$ is big enough to justify the ``unimaginable'' attribute.
\end{example}
We recall that the Graham number $G$ is defined in Knuth's notation (see \cite{UN} for a wide discussion on hyper-operations) by the following recursive notation:
\begin{align*}
a&\uparrow b=a*a*a\ldots*a\text{ [$b$ times]}=a^b\\
a&\uparrow\uparrow b=a\uparrow a\uparrow a\ldots\uparrow a\text{ [$b$ times]}\\
a&\uparrow^k b=a\uparrow^{k-1}a\uparrow^{k-1}a\uparrow^{k-1}a\ldots\text{ [$b$ times]}\\
g_0&=4\\
g_1&=3\uparrow\uparrow\uparrow\uparrow 3\\
g_k&=3\uparrow^{g_{k-1}}3\\
G&=g_{64}
\end{align*}
This literally unimaginable number was involved in the first proof for Graham's problem in Conway's theory, even though it has been recently substituted by the smaller (but still insanely big) number $2\uparrow\uparrow2\uparrow\uparrow2\uparrow\uparrow9$.
\begin{example}[Last digits of Graham's number]
We now consider the problem of computing the last digits of $G=g_{64}$.\par\medskip
It can be proved that the following algorithm:
$$a_0=3\text{; }x_i=a_i\text{ (mod $10^i$)}\text{; }a_{i+1}=3^{x_i}\text{; }x=\lim_{i\rightarrow+\infty}x_i$$
converges in $\Z_{10}$ and that also $g_\infty:=\lim_{i\rightarrow\infty}g_i$ exists and $x=g_\infty$.\par\smallskip
This number can be seen to be the fixed point for the equation $x=3^x$.\par\medskip
Observing that the $10$-adic difference $G-g_\infty$ ends with an ``unimaginable'' number of zeroes, their last digits are the same for all practical purposes.
\end{example}
\begin{remark}
More generally, an infinite \emph{tetration} $k\uparrow\uparrow\infty:=\lim_{j\to\infty}k\uparrow\uparrow j$ always converges in any $n$-adic ring (see also \cite{TT}), as it must converge for any $p$-adic value: either $p$ divides $k$, thus the sequence converges to $0$, or they are coprime and the convergence is similar to the one of \emph{continued exponentials} (see \cite{CF.NA}).\par\medskip
This number too can be viewed as a fixed point:
$$x=k^x\rightarrow x=k\uparrow\uparrow\infty:=\lim_{i\rightarrow\infty}k\uparrow\uparrow i.$$
This is of course a generalization for $g_\infty=3\uparrow\uparrow\infty$.
\end{remark}

\subsection*{Acknowledgements}
I want to thank my family and friends for supporting me, and also my former advisor R. Dvornicich for his suggestion on participating to JMM. Also a thank goes to my research funder G. d'Atri for his suggestions and support on this work.\par\medskip
This work has been partially supported by POR Calabria FESR-FSE 2014--2020, with the grant for research project ``IoT\&B'', CUP J48C17000230006.

\end{document}